\documentclass{file}
\usepackage{pat}
\usepackage{paralist}
\usepackage{dsfont}  
\usepackage[utf8x]{inputenc}
\usepackage{skull}
\usepackage{paralist}
\usepackage{centernot}
\usepackage{mathtools}
\usepackage{thmtools}
\usepackage{thm-restate}
\usepackage{hyperref}
\usepackage{cleveref}
\usepackage{stmaryrd}
\usepackage{graphicx}
\usepackage[noend]{algorithmic}
\usepackage{bbm}  
\usepackage{listings}
\usepackage{pifont}%
\usepackage[normalem]{ulem}
\usepackage{cancel}

\usepackage{todonotes}

\usepackage{mathdots}
\usepackage{bm}

\usepackage{etoolbox}

\usepackage{float}
\usepackage{tabularx}
\usepackage{array}
\usepackage{multirow}
\usepackage{makecell}
\definecolor{maroon}{rgb}{0.5, 0.0, 0.0}
\definecolor{darkblue}{rgb}{0.0, 0.0, 0.55}

\newcommand{\cmark}{\ding{51}}%
\newcommand{\xmark}{\ding{55}}%

\newcommand{\bigzero}{\mbox{\normalfont\large\bfseries 0}}

\usepackage[longnamesfirst,numbers,sort&compress]{natbib}

\usepackage[mathlines]{lineno}
\newcommand*\patchAmsMathEnvironmentForLineno[1]{%
 \expandafter\let\csname old#1\expandafter\endcsname\csname #1\endcsname
 \expandafter\let\csname oldend#1\expandafter\endcsname\csname end#1\endcsname
 \renewenvironment{#1}%
    {\linenomath\csname old#1\endcsname}%
    {\csname oldend#1\endcsname\endlinenomath}}%
\newcommand*\patchBothAmsMathEnvironmentsForLineno[1]{%
 \patchAmsMathEnvironmentForLineno{#1}%
 \patchAmsMathEnvironmentForLineno{#1*}}%
\AtBeginDocument{%
\patchBothAmsMathEnvironmentsForLineno{equation}%
\patchBothAmsMathEnvironmentsForLineno{align}%
\patchBothAmsMathEnvironmentsForLineno{flalign}%
\patchBothAmsMathEnvironmentsForLineno{alignat}%
\patchBothAmsMathEnvironmentsForLineno{gather}%
\patchBothAmsMathEnvironmentsForLineno{multline}%
}

\definecolor{brightmaroon}{rgb}{0.76, 0.13, 0.28}
\definecolor{linkblue}{rgb}{0, 0.337, 0.227}
\newcommand{\cay}{\mathop{\mathsf{Cay}}}

\newcommand{\rededge}{\mathord{\,\color{red}\vrule width 12pt height 3pt depth -1.5pt}\,}
\newcommand{\blueedge}{\mathord{\,\color{blue}\vrule width 12pt height 3pt depth -1.5pt}\,}
\newrobustcmd{\onesub}{\mathord{\includegraphics{one-sub}}}
\newrobustcmd{\leftup}{\mathord{\includegraphics{left-up}}}

\newtheorem{case}{Case}

\makeatletter
\newcommand{\xMapsto}[2][]{\ext@arrow 0599{\Mapstofill@}{#1}{#2}}
\def\Mapstofill@{\arrowfill@{\Mapstochar\Relbar}\Relbar\Rightarrow}
\makeatother

\title{\MakeUppercase{On Matrix Product Factorization of graphs}}
\author{Farzad Maghsoudi\thanks{Department of Mathematics and Computer Science, University of Lethbridge, Lethbridge, AB, Canada.} \,\,\,\,\,Babak Miraftab\thanks{School of Computer Science, Carleton University, Ottawa, ON, Canada.}  \,\,\,\,\, Sho Suda\thanks{Department of Mathematics, National Defense Academy of Japan, Yokosuka, Kanagawa 239-8686,
Japan}}

\date{}

\begin{document}

\maketitle

\begin{abstract}
In this paper, we explore the concept of the ``matrix product of graphs," initially introduced by Prasad, Sudhakara, Sujatha, and M. Vinay. This operation involves the multiplication of adjacency matrices of two graphs with assigned labels, resulting in a weighted digraph.
Our primary focus is on identifying graphs that can be expressed as the graphical matrix product of two other graphs. Notably, we establish that the only complete graph fitting this framework is $K_{4n+1}$, and moreover the factorization is not unique. 
In addition,
the only complete bipartite graph that can be expressed as the graphical matrix product of two other graphs is $K_{2n,2m}$
Furthermore, we introduce several families of graphs that exhibit such factorization and, conversely, some families that do not admit any factorization such as wheel graphs, friendship graphs, hypercubes and paths.
\end{abstract}

\section{Introduction}
The concept of the ``matrix product of graphs" was introduced in the work by Prasad, Sudhakara, Sujatha, and M. Vinay, in \cite{Manjunatha}. In essence, the matrix product of graphs is an operation that merges two graphs into a new one by performing a matrix multiplication on their adjacency matrices. 
To carry out this operation, we establish labels(orderings) for both graphs, and their adjacency matrices are multiplied using these labels, resulting in a weighted digraph.
In this paper, we focus on pairs of graphs $(G, H)$ for which the matrix product is symmetric, and all edge weights are set to 1. Such a pair is referred to as a graphical pair, as described in \Cref{graphical_pair} with more details.
It is important to note that the matrix product of graphical pairs is commutative. This commutativity arises from the symmetry of the adjacency matrices of both graphs, $G$ and $H$.
$$A(G)A(H)=A(G)^tA(H)^t=(A(H)A(G))^t=A(H)A(G)$$
The matrix product of graphical pairs has recently gained attention, see \cite{Z2, Bhat, Arathi} and in particular see the survey \cite{survey}.
However, the origins of this product trace back to a problem of finding commutating pairs of graphs proposed by Akbari and Herman in \cite{Akbari} and later again by Akbari, Moazami, and Mohammadian in \cite{Akbari2}.
We say that two graphs $G_1$ and $G_2$ with the same vertex set \emph{commute} if their adjacency matrices commute. 
They proved that
\begin{thm}
    Let $n\geq 2$ be an integer. Then the following items hold:
    \begin{itemize}
        \item \cite[Theorem 2]{Akbari2} The complete graph $K_n$ is decomposable into commuting perfect matchings and $2$-regular graphs.
        \item \cite[Theorem 1]{Akbari} The graph $K_{n,n}$ is decomposable into commuting perfect matchings if and only if $n$ is a power of $2$.
    \end{itemize}
\end{thm}

\noindent For further applications, readers are referred to \cite{commuting, Somodi, de2009necessary}.
Motivated by the decomposition problem of $K_{n,n}$ and $K_n$ into commuting perfect matchings in \cite{Akbari} and \cite{Akbari2}, our research focuses on the matrix product decomposition of a connected graph $X$ on $n$ vertices into two commuting graphs on $n$ vertices. 
More precisely we ask the following question:

\begin{prb}\label{main}
Determine the family $\mathcal{X}$ of all graphs such that the adjacency matrix of every $X$ in $\mathcal X$, denoted as $A(X)$, can be expressed as the product of the adjacency matrices of two other graphs, $G$ and $H$, i.e. $A(X) = A(G) A(H)$.
\end{prb}

\noindent It is worth mentioning that Bosch in \cite{Bosch} shows that every real square matrix $X$ can be expressed as a matrix product of two real symmetric matrices, denoted as $G$ and $H$, see \cite[Theorem 2]{Bosch}. 
However, it is important to note that there is no assurance that $G$ and $H$ correspond to the adjacency matrices of graphs.
Let us define that a graph X is a \emph{matrix product factorizable} if $X$ can be expressed as a matrix
product of adjacency matrices of two graphs, see \Cref{MPF} for a more rigorous definition.
Arathi Bhat,  Sudhakara, and  Vinay in \cite{Bhat} showed that generalized wheel graphs are matrix product factorizable, see \cite[Theorem 2.1, Remark 2.2]{Bhat}.
In this paper, we investigate \Cref{main} by focusing on complete graphs and complete bipartite graphs. For instance, we show that the only complete graph that belongs to $\mathcal{X}$ is $K_{4n+1}$.
 
\begin{restatable}{thm}{completegraphs}
\label{thm:completegraphs}
     The complete graph $ K_{n} $ can be expressed as the matrix product of two graphs if and only if $ n = 4k+1 $.
\end{restatable}

\noindent In addition, we show that the only complete bipartite graph that belongs to $\mathcal{X}$ is $K_{2n,2m}$.
\begin{restatable}{thm}{completebipartitegraphs}
\label{thm:completebipartitegraphs}
    The complete bipartite graph $ K_{m,n} $ can be expressed as the matrix product of two graphs if and only if $m$ and $n$ are even.
\end{restatable}

\noindent Furthermore, we explore various other graph families, such as wheel graphs, paths, circulant graphs, and others, to determine whether they have matrix product factorization properties.
Last but not least,  we direct readers to the \cref{summary-table}, for a list of our results.


\section{Preliminaries}
For terminology and notation not found in this paper, we refer readers to~\cite{godsil01}. 
A \emph{weighted digraph} is a digraph in which a number (the weight) is assigned to each arc (edge).
A simple graph is a graph without multiple edges.

\noindent Now, we define the matrix product of two graphs in the most generalized manner.
\begin{defn}
    Let $ G $ and $ H $ be two graphs on $ n $ vertices, with labelings (bijection) $ g\colon V(G) \rightarrow \{1,2,\ldots,n\} $, and $ h\colon V(H) \rightarrow \{1,2,\ldots,n\} $. The \emph{matrix product} of $G$ and $H$ with respect to $g$ and $h$ denoted by $ G {}_{g}\cdot{}_{h} H $ is a weighted digraph with  the adjacency matrix $ A(G)A(H) $.
\end{defn}

\noindent It is worth mentioning that the labelings of $G$ and $H$ play an important role here. 
In the following example, we illustrate that altering the labels results in distinct digraphs.

\begin{exa}
Let $ G $ and $ H $ be graphs on $ 6 $ vertices depicted in \cref{ex22}.
\begin{figure}[H]
    \centering
\includegraphics[scale=1.1]{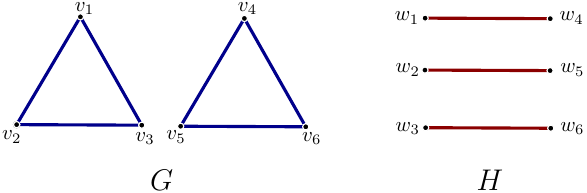}
    \caption{$ G $ and $ H $}
    \label{ex22}
\end{figure}

\noindent Define the following labelings:
\[
g\coloneqq \begin{pmatrix} 
	v_1 & v_2 & v_3 & v_4 & v_5 & v_6 \\
	1 & 2 & 3 & 4 & 5 & 6 \\
	\end{pmatrix} , \textit{ and }  h\coloneqq \begin{pmatrix} 
	w_1 & w_2 & w_3 & w_4 & w_5 & w_6 \\
	1 & 2 & 3 & 4 & 5 & 6 \\
	\end{pmatrix}.
 \] Then, we have
\begin{align*}
    A(G)A(H) = \begin{bmatrix} 
	0 & 1 & 1 & 0 & 0 & 0 \\
	1 & 0 & 1 & 0 & 0 & 0 \\
	1 & 1 & 0 & 0 & 0 & 0 \\
	0 & 0 & 0 & 0 & 1 & 1 \\
    0 & 0 & 0 & 1 & 0 & 1 \\
	0 & 0 & 0 & 1 & 1 & 0 \\
	\end{bmatrix} \begin{bmatrix} 
	0 & 0 & 0 & 1 & 0 & 0\\
	0 & 0 & 0 & 0 & 1 & 0\\
	0 & 0 & 0 & 0 & 0 & 1\\
	1 & 0 & 0 & 0 & 0 & 0\\
    0 & 1 & 0 & 0 & 0 & 0\\
	0 & 0 & 1 & 0 & 0 & 0\\
	\end{bmatrix} = \begin{bmatrix} 
	0 & 0 & 0 & 0 & 1 & 1\\
    0 & 0 & 0 & 1 & 0 & 1\\
    0 & 0 & 0 & 1 & 1 & 0\\
    0 & 1 & 1 & 0 & 0 & 0\\
    1 & 0 & 1 & 0 & 0 & 0\\
    1 & 1 & 0 & 0 & 0 & 0\\
	\end{bmatrix},
 \end{align*}
 which is the corresponding adjacency matrix to graph $ G {}_{g}\cdot{}_{h} H $~(see~\cref{ex23}).
Now define, 
\[
g'\coloneqq \begin{pmatrix} 
	v_1 & v_2 & v_3 & v_4 & v_5 & v_6\\
	1 & 3 & 5 & 4 & 2 & 6\\
	\end{pmatrix}, \textit{ and } h'\coloneqq \begin{pmatrix} 
	w_1 & w_2 & w_3 & w_4 & w_5 & w_6 \\
	1 & 2 & 3 & 4 & 6 & 5 \\
	\end{pmatrix},
\] 
  as labelings on vertices of $ G $ and $ H $, respectively. Then, we have
  \begin{align*}
    A(G)A(H) = \begin{bmatrix} 
	0 & 0 & 1 & 0 & 1 & 0 \\
	0 & 0 & 0 & 1 & 0 & 1 \\
	1 & 0 & 0 & 0 & 1 & 0 \\
	0 & 1 & 0 & 0 & 0 & 1 \\
 	1 & 0 & 1 & 0 & 0 & 0 \\
	0 & 1 & 0 & 1 & 0 & 0 \\
	\end{bmatrix} \begin{bmatrix} 
	0 & 0 & 0 & 1 & 0 & 0\\
	0 & 0 & 0 & 0 & 0 & 1\\
	0 & 0 & 0 & 0 & 1 & 0\\
	1 & 0 & 0 & 0 & 0 & 0\\
 	0 & 0 & 1 & 0 & 0 & 0\\
	0 & 1 & 0 & 0 & 0 & 0\\
	\end{bmatrix} = \begin{bmatrix} 
	0 & 0 & 1 & 0 & 1 & 0\\
    1 & 1 & 0 & 0 & 0 & 0\\
    0 & 0 & 1 & 1 & 0 & 0\\
    0 & 1 & 0 & 0 & 0 & 1\\
    0 & 0 & 0 & 1 & 1 & 0\\
    1 & 0 & 0 & 0 & 0 & 1\\
	\end{bmatrix},
 \end{align*}
 which is the corresponding adjacency matrix to $ G {}_{g'}\cdot{}_{h'} H $~(see \cref{ex23}). Clearly, the two graphs in~\cref{ex23} are different. 
\end{exa}
\begin{figure}[H]
    \centering
\includegraphics[scale=0.75]{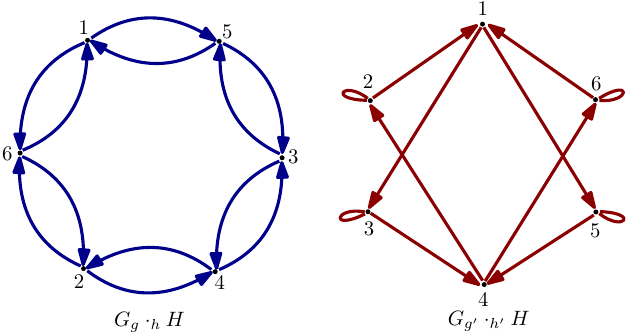}
    \caption{$ G{}_{g}\cdot{}_{h}H $ and $ G{}_{g'}\cdot{}_{h'}H $}
    \label{ex23}
\end{figure}

\noindent We are interested in digraphs where every edge is assigned a weight of 1. In pursuit of this objective, we establish the following definition:

\begin{defn}
   Let $ G $ and $ H $ be two graphs on $ n $ vertices, with labelings $ g\colon V(G) \rightarrow \{1,2,\ldots,n\} $, and $ h\colon V(H) \rightarrow \{1,2,\ldots,n\} $. A pair of $((G,g),(H,h))$ of graphs without loops is called \emph{graphical} if $A(G)A(H)$ is a symmetric $(0,1)$-matrix. 
\end{defn}

\noindent 
\noindent To ensure the resulting graph of a matrix product is a simple graph, we modify the previous product as follows:
\begin{defn}\label{graphical_pair}
    Let  $((G,g),(H,h))$ be a graphical pair. Then the \emph{graphical matrix product} $(G,g)$ and $(H,h)$ denoted  $ G {}_{g}\ast_h H $ is the underlying graph of $ G {}_{g}\cdot{}_{h} H $ i.e. replace each $2$-cycle with a single edge.
\end{defn}
\noindent We observe that the matrix product of a graphical pair $((G,g),(H,h))$ is commutative, specifically $ G {}_{g}\ast_h H = H {}_{h}\ast_g G$. This is due to the symmetry of the adjacency matrices of $G$ and $H$ concerning $g$ and $h$, as exemplified by the equation:
\[ A(G)A(H) = A(G)^tA(H)^t = (A(H)A(G))^t = A(H)A(G) \]
Thus, we arrive at the following conclusion.

\begin{rem}\label{simdiag}
Let a simple graph $X$ can be expressed as $G {}_{g}\ast_h H$.
Then the adjacency matrices $A(G)$ and $A(H)$ with respect to $g$ and $h$ are simultaneously diagonalizable. 
\end{rem}
\noindent A natural question that can be raised is whether there exists a pair of graphs $(G, H)$ for which any matrix product of $G$ and $H$ yields the same result.
Later, we show that such a pair does not exist, see \cref{loop}.
We next introduce an edge colouring for graphs and establish the relationship between this colouring and the matrix product of graphs.
Recall that edge colouring involves the assignment of colours to the edges of a graph. When we use $k$ colours on the edges, we say that $G$ has a $k$-edge colouring. 
It is important to note that $k$-edge colouring is not necessarily proper, meaning that adjacent edges may have the same colour.

\begin{defn}
Let $G$ be a graph with a $2$-edge colouring.
A cycle of length $4$ is called a \textit{coloured diamond} if its edges are alternately coloured.
A $2$-edge colouring of $G$ satisfies the \emph{diamond condition} if, for every pair $\{u,v\}$ of vertices in $G$ such that they are connected with a path of length $2$ with alternating colours, the vertices $\{u,v\}$ belong to exactly one coloured diamond.
\end{defn}

\noindent In the other words, the diamond condition says if $w\in N(u)\cap N(v)$ such that $uw$ is blue(red) and $wv$ is red(blue), then there exists one vertex $z$ in $N(u)\cap N(v)$ such that $uz$ is red(blue) and $zv$ is blue(red).
 
\begin{exa}
In \cref{2-edge colouring}, we provided an example of a 2-edge colouring of $ K_5 $ satisfying the diamond condition.   
\end{exa}

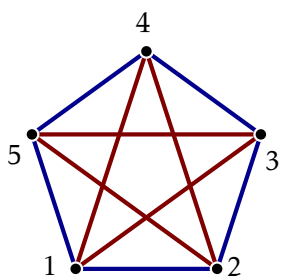
\begin{figure}[H]
    \centering
\begin{tikzpicture}[scale=0.8]	

\begin{scope}[rotate=18]
    \foreach \i in {1,...,5} {
    \node[circle,fill, inner sep=2pt,white] (\i) at (\i*72:2) {};
    \node[circle,fill, inner sep=1.pt] (\i) at (\i*72:2) {};}
    \draw[ultra thick,darkblue] (1) -- (2);
   \draw[ultra thick,darkblue] (2) -- (3);
    \draw[ultra thick,darkblue] (3) -- (4);
    \draw[ultra thick,darkblue] (4) -- (5);
   \draw[ultra thick,darkblue] (5) -- (1);
   \draw[ultra thick,maroon] (4) -- (2);
   \draw[ultra thick,maroon] (2) -- (5);
   \draw[ultra thick,maroon] (5) -- (3);
   \draw[ultra thick,maroon] (3) -- (1);
    \draw[ultra thick,maroon] (1) -- (4);
    
    \draw (0.55,1.9) node[label=above:$4$]{};
    \draw (2.2,0) node[label=below:$3$]{};
    \draw (0.90,-1.95) node {$2$};
    \draw (-2,-1) node {$1$};
    \draw (-2,0.95) node {$5$};

  \foreach \i in {1,...,5} {

    \node[circle,fill, inner sep=2pt,white] (\i) at (\i*72:2) {};
        \node[circle,fill, inner sep=1.4pt] (\i) at (\i*72:2) {};
    }    
\end{scope}
    
\end{tikzpicture}
    \caption{A 2-edge colouring with the diamond condition}\label{2-edge colouring}
\end{figure}

\begin{nota}\cite{Rosen}
Given two graphs $G$ and $ H $ with labelings $g\colon V(G)\to \{1,\ldots,n\}$ and $h\colon V(H)\to \{1,\ldots,n\}$, respectively, the disjoint union of $G$ and $H$  denoted by $ G\oplus H$ is a graph whose vertex set is $\{1,\ldots,n\}$ and edge set is the disjoint unions of $E(G)\cup E(H)$.
\end{nota}

\noindent Let each edge of $G$ and $H$ be assigned a colour $c_1$ and $c_2$, respectively.
Then, the edges of the graph $E(G\oplus H)$ can be classified into two sets based on their colours.
In order to study the matrix multiplication of $G$ and $H$, a more thorough analysis of the edge colouring of $G\oplus H$ is required.
Prasad, Sudhakara, Sujatha, and Vina proved the following lemma.

\begin{lem}{\cite[Theorem 1]{Manjunatha}}\label{lem 3}
Let $G$ and $ H $  be graphs without loops with labelings $g\colon V(G)\to\{1,\ldots,n\}$, and $ h
\colon V(H)\to\{1,\ldots,n\}$.
Then
\begin{enumerate}
    \item $ G {}_{g}\cdot{}_{h} H $ has no loop if and only if $G\oplus H$ is simple (has no multiple edges),
    \item $((G,g),(H,h))$ is graphical if and only if there is a $2$-edge colouring of $G\oplus H$ satisfying the diamond condition.
\end{enumerate}  
\end{lem}

\begin{lem}{\cite[Theorem 7]{Manjunatha}}\label{thm 4}
    Let $((G,g),(H,h))$ be a graphical pair of graphs, such that $G{}_{g}\ast{}_{h} H $ has no loops. Then,
    \begin{align*}
        \textnormal{deg}_{G{}_{g}\ast{}_{h} H}v = \textnormal{deg}_{G}v\cdot\textnormal{deg}_{H}v.
    \end{align*}
\end{lem}


\begin{cor}\label{thm 5}
    Let $((G,g),(H,h))$ be a graphical pair of graphs, such that $G{}_{g}\ast{}_{h} H $ has no loops. 
    If $G$  and $H$ are regular graphs, then $ G{}_{g}\ast{}_{h} H $ is regular.
\end{cor}

\begin{lem}{\cite[Corollary 2]{Manjunatha}}\label{deg}
    Let $((G,g),(H,h))$ be a graphical pair of graphs, such that $G{}_{g}\ast{}_{h} H $ has no loops.
    If $u$ and $v$ are connected in $H$, then $\emph{deg}_G(u)=\emph{deg}_G(v)$.
\end{lem}

\section{Factorization of Graphs via Matrix Product}

A graph $X$ without loops is a \emph{matrix product factorizable} if $X$ can be expressed as a matrix product of adjacency matrices of two graphs.
More precisely we have:
\begin{defn}\label{MPF}
    A simple graph $X$ is a \emph{matrix product factorizable} or an MPF graph if there exists a graphical pair $((G,g),(H,h))$ of simple graphs, and also a labeling of $X$ such that $A(G{}_{g}\ast_{h} H)=A(X)$.
\end{defn}

\noindent The definition provided above naturally leads to the following question: which graphs are MPF? In the following, we will present a couple of lemmas that assist us in showing that certain families of graphs do not possess the MPF graph property.
First, we show that the property of being MFD for a given graph $X$ is independent of fixing a labeling for $X$.

\begin{lem}\label{anylabel}
    Assume a graph $ X $ is an MPF graph. Then for any adjacency matrix $ A(X) $ of $X$, there is a graphical pair of graphs $((G,g),(H,h))$ such that $ A(X) = A(G)A(H) $. 
\end{lem}
    \begin{proof}
        Let $ X $ be a graph on $ n $ vertices. Since $ X $ is an MPF graph there are labelings of $ G $, $ H $, and $ X $ such that $ A(G)A(H) = A(X)$. 
        We now assume that $Z$ is the adjacency matrix $X$ with an arbitrary labeling.
        Clearly, $ Z = P^{-1}A(X)'P $ for some permutation matrix $ P $. This implies that $ A(G)A(H) = PZP^{-1} $. So, we have $ (P^{-1}A(G)P)(P^{-1}A(H)P) = A(X) $, as desired.
    \end{proof}

\begin{thm}\cite[Theorem 4] {Manjunatha}\label{odd-edges}
   If $ X $ is an MPF graph, then $ X $ has an even number of edges.
\end{thm}

\begin{cor}\label{C_odd}
    The following graphs are not MPF:
    \begin{itemize}
        \item The cycle $ C_{2k+1} $.
        \item The $2n$-vertex path, $P_{2n}$. 
        \item $ 2n $-vertices trees.
        \item The web graph $ W_{2n+1,2} $.
    \end{itemize}
\end{cor}

\begin{lem}\label{lem 19}
    Let  $ X $ be a graph without loops on $ n\geq 3 $ vertices with degree sequence $(d_1,\ldots,d_n)$, such that $ d_i = p_1^{j_1}\cdots p_{\ell}^{j_{\ell}} $, where $ p_i $'s are distinct primes and $d_k>0$ for every $k$. If there exists $ p_i \in \{p_1,\ldots,p_{\ell}\} $, such that $ \gcd(p_i,d_k) = 1 $, for each $k\in \{1,\ldots,n\}\sm \{i\}$, then $ X $ is not an MPF graph.
\end{lem}

\begin{proof}
    Assume to the contrary that $X$ is an MPF graph.
    So there is a labeling $ \gamma\colon V(X) \to\{1,\ldots,n\}$ and  
    let $ G $ and $ H $ be two simple graphs with labelings $ g\colon V(G)\to \{1,\ldots,n\} $ and $ h\colon V(H)\to \{1,\ldots,n\} $, respectively, such that $ G{}_{g}\ast{}_{h} H = X $.
    Let $ v\in V(X) $ be the vertex that $\textnormal{deg}_{X}(v) = d_i $. Also, let  $ E(G) $ be coloured by blue and $ E(H) $ be coloured by red. 
    By \cref{thm 4}, there exist a vertex $ v\in V(G\oplus H) $ such that $ \textnormal{deg}_{X}(v) = \textnormal{deg}_{G}(v)\cdot \textnormal{deg}_{H}(v)  $. 
    Since $\textnormal{deg}_{X}(v)>0$, we conclude that $\textnormal{deg}_{G}(v)>0$ and $\textnormal{deg}_{H}(v)>0$.
    Considering the vertex $ v $ in $ G \oplus H $, then there are $ \textnormal{deg}_{G}(v) $ blue edges and $ \textnormal{deg}_{H}(v)$ red edges incident with $ v $. Let $ \{v,v'\} $ be one of these red edges in $ G \oplus H $. 
    Thus, by \cref{deg} we must have $ \textnormal{deg}_{G}(v)=\textnormal{deg}_{G}(v') $. 
    Since  $ p_i \vert~\textnormal{deg}_{G}(v) $, we infer that $ p_i \vert~\textnormal{deg}_{G}(v') $.
    This implies that $ \gcd(\textnormal{deg}_{X}(v'),p_i) \neq 1 $ which contradicts the assumption that $ \gcd(\textnormal{deg}_{X}(v'),p_i) = 1 $, for degree of all vertices of $ X $ except $ v $.
\end{proof}

\begin{cor}\label{lem 14} The following graphs are not MPF:
\begin{itemize}
    \item The wheel graph $ W_{2n+1} $.(a wheel graph is a graph formed by connecting a single universal vertex to all vertices of a cycle.)
    \item The friendship graph $ F_n $, if $ n\neq 2^k $ for any integer $ k\geq 1 $. (The friendship graph $ F_n $ can be constructed by joining $ n $ copies of $ 
C_3 $ with a common vertex.)
\end{itemize}
    
\end{cor}

\begin{lem}\label{lem 20}
    Let $X$ be an MPF graph without loops.
    Then the multiplicity of a prime degree vertex in the degree sequence is even.
\end{lem}

\begin{proof}
     Assume to the contrary that $X$ is an MPF graph into $G$ and $H$ with the labeling $ \gamma\colon V(X)\to \{1,2\ldots,n\} $ for $ X $, and labelings $ g\colon V(G)\to \{1,2,\ldots,n\} $ and $ h\colon V(H)\to \{1,2,\ldots,n\} $, respectively. Since $ p $ is prime with odd multiplicity in $ X $, then by \cref{thm 4}, either $ G $ or $ H $ contains an odd number of vertices with degree $ p $. 
     Let $Z$ be the set of all vertices in $X$ whose degrees are $p$.
     Suppose $ E(G) $ is coloured blue and $ E(H) $ is coloured red. 
     Now let $ v\in Z $  and assume that $ \deg_{G}(v) =p $ and  $ \deg_{H}(v) =1 $, see \cref{thm 4}. 
     This means $ v $ has $ p $ neighbours with blue edges in $ G\oplus H $ and has exactly 1 neighbour with a red edge in $ G\oplus H $. 
     Let $ vv_1 $ be this specific red edge in $ G \oplus H $. Then by \cref{lem 3}, the vertex $ v_1 $ must have $ p $ neighbours with blue edges and exactly 1 neighbour with a red edge in $G\oplus H $. So, $ \{v,v_1\} $ is a red edge in $ G\oplus H $ such that $ \deg_{G}(v) = \deg_{G}(v') = p $. 
     Therefore we deduce that $v_1\in Z$.
     Now we delete both $v$ and $v_1$ from $Z$ and repeat the same argument we a new vertex $u\in Z\sm \{v,v'\}$.
     If there are odd vertices in $ X $ with degree $ p $($|Z|$ is an odd number), then by using the above argument, we end up with exactly one vertex~$ w $ of the same degree, then it is not possible.
 \end{proof}

 \begin{cor}\label{odd-path}
     The following graphs are not MPF:
    \begin{itemize}
    \item $ (2n+1) $-vertex path graph, $ P_{2n+1} $.
        \item $ K_{p,q} $, where $ p $, $ q $ are distinct primes.
    \end{itemize}
\end{cor}

\begin{prop}\label{prop:path}
    The path graph $P_n$ is not MPF.
\end{prop}

\begin{proof}
    It follows from \cref{C_odd} and \cref{odd-path}. 
\end{proof}

\begin{prop}{\cite[Corollary 3]{Manjunatha}}\label{odd-degree-vertex}
    Let $ X $ be a graph on $ n>3 $ vertices, where $ n $ is even. If there exists a vertex $ x \in V(X) $, such that $ \deg(x) = n-1 $, then $ X $ is not an MPF graph. 
\end{prop}

\begin{cor}\label{W_2n-K_2n}
    The following graphs are not MPF:
    \begin{itemize}
        \item The complete graph $ K_{2n} $,
        \item The wheel graph $ W_{2n} $.
    \end{itemize}
\end{cor}

\begin{thm}\label{wheel-graph}
    The wheel graph $ W_n $ is not an MPF graph.
\end{thm}

\begin{proof}
    If $ n $ is odd, then \cref{lem 14} applies. If $ n $ is even, then \cref{W_2n-K_2n} applies.
\end{proof}

\begin{prop}\label{loop}
There are no graphs $G$ and $H$ without loops such that, for every graphical pair $((G,g),(H,h))$ of $G$ and $H$, the matrix product $G_{g} \ast_{h} H$ always yields a graph $X$ without isolated vertices and loops.
\end{prop}

\begin{proof}
If either $G$ or $H$ is a complete graph, according to \cref{lem 3}, the other graph must have no edges, resulting in $X$ having an isolated vertex.

\noindent \textbf{Case 1:} Suppose there exists a vertex $u$ in $G$ with a degree of at least $2$. Let $g$ be a labeling of $G$ such that the labels of $u$ and its neighbors are ${1,\ldots,\textnormal{deg}_G(u)+1}$.
If $H$ has a vertex $v$ with a degree of at least $2$, then we can consider a labeling $h$ of $H$ such that the labels of $v$ and its neighbors are ${1,\ldots,\textnormal{deg}_H(v)+1}$. In this case, it is evident that $A(G)A(H)$ has an entry with a value greater than $1$. 
Consequently, it implies that the degree of each vertex is exactly $1$, leading to the conclusion that $H$ is a perfect matching. It is worth noting that, in this case, the number of vertices in both $G$ and $H$ must be even. As a result, according to \cref{odd-degree-vertex}, the degree of each vertex is at most $n-2$.
Now, let's consider the following labelings for $G$ and $H$:
Let $u$ be a vertex of $G$. Define the labeling $g$ of $G$ such that the labels of $u$ and its neighbors are ${1,\ldots,\textnormal{deg}_G(u)+1}$. Additionally, define the labeling $h$ of $H$ such that $i\sim 2n+1-i$.
In this case, we observe that $A(G)A(H)$ has a zero row corresponding to an isolated vertex in $X$.

\noindent \textbf{Case 2:} We now assume that both $H$ and $G$ are perfect matchings. In this scenario, it is not difficult to find labelings $g$ and $h$ for $G$ and $H$ respectively, such that $G_{g} \ast_{h} H$ has an isolated vertex.
 \end{proof}

\begin{lem}
    Let $ X $ be a graph on $ 2n+1 $ vertices with the degree sequence $ \{2(2k_i+1)\}_{i=1}^{2n+1} $. Then $ X $ is not an MPF graph.
\end{lem}

\begin{proof}
Assume, for the sake of contradiction, that $ X $ is a MPF graph. 
There exists a graphical pair $ ((G,g),(H,h)) $, consisting of graphs without loops, such that $ G_{g}\ast{}_{h}H = X $. 
Let the degree sequences of $ G $ and $ H $ be denoted by $d_1,d_2,\ldots,d_{2n+1}$ and $ d'_1,d'_2,\ldots,d'_{2n+1} $, respectively. Additionally, let $ \ell_{e} $ be the number of vertices with even degree in $ G $, and $ \ell_{o} $ be the number of vertices with odd degree in $ G $. Clearly, $ \ell_{e} $ is odd and $ \ell_{o} $ is even. Now according to \cref{thm 4}, there are $ \ell_{e} $ vertices with odd degree in $ H $ and $ \ell_{o} $ vertices with even degree in $ H $, which is a contradiction, since $ \ell_{o} $ is odd. 
\end{proof}

\begin{cor}\label{K_4n+3}
    The complete graph $ K_{4n+3} $ is not an MPF graph.
\end{cor}

\subsection{Complete graphs}
This subsection is dedicated to addressing the following problem:
\begin{prb}
For which values of $n$ is the complete graph $K_n$ an MPF graph?
\end{prb}

\begin{exa}\label{example 10}
    The complete graph $ K_9 $ is an MPF graph. Let $ G $ be a graph with $ g\colon V(G) \to \{0,1,\ldots,8\} $, and $ E(G) = \sqcup_{k=1}^2 E_k $, where $ E_k = \{i \,\blueedge \,(i+2k-1) \pmod{9}\mid i\in \{0,\ldots,8\}\} $, and $ H $ be a graph with  $ h\colon V(H) \to \{0,1,\ldots,8\} $, and $ E(H) = \{i\,\rededge\, (i+4) \pmod{9}\mid i\in \{0,\ldots,8\}\} $~(see~\cref{ex2}). 
    It is not hard to verify that $ H $ is a $ 
C_9 $ cycle. Then we have the following two adjacency matrices:
\begin{align*}
    A(G) = \begin{bmatrix} 
	0 & 1 & 0 & 1 & 0 & 0 & 1 & 0 & 1 \\
	1 & 0 & 1 & 0 & 1 & 0 & 0 & 1 & 0 \\
    0 & 1 & 0 & 1 & 0 & 1 & 0 & 0 & 1 \\
	1 & 0 & 1 & 0 & 1 & 0 & 1 & 0 & 0 \\
 0 & 1 & 0 & 1 & 0 & 1 & 0 & 1 & 0 \\
 0 & 0 & 1 & 0 & 1 & 0 & 1 & 0 & 1 \\
  1 & 0 & 0 & 1 & 0 & 1 & 0 & 1 & 0 \\
  0 & 1 & 0 & 0 & 1 & 0 & 1 & 0 & 1 \\
  1 & 0 & 1 & 0 & 0 & 1 & 0 & 1 & 0 \\
	\end{bmatrix},~A(H) = \begin{bmatrix} 
	0 & 0 & 0 & 0 & 1 & 1 & 0 & 0 & 0 \\
	0 & 0 & 0 & 0 & 0 & 1 & 1 & 0 & 0 \\
    0 & 0 & 0 & 0 & 0 & 0 & 1 & 1 & 0 \\
	0 & 0 & 0 & 0 & 0 & 0 & 0 & 1 & 1 \\
 1 & 0 & 0 & 0 & 0 & 0 & 0 & 0 & 1 \\
 1 & 1 & 0 & 0 & 0 & 0 & 0 & 0 & 0 \\
  0 & 1 & 1 & 0 & 0 & 0 & 0 & 0 & 0 \\
  0 & 0 & 1 & 1 & 0 & 0 & 0 & 0 & 0 \\
  0 & 0 & 0 & 1 & 1 & 0 & 0 & 0 & 0 \\
	\end{bmatrix},\\
\end{align*}
which implies that $A(G)A(H) = J_9-I_9$, and this is the adjacency matrix of $ K_9 $.

\begin{figure}[H]
    \centering
\begin{tikzpicture}[scale=0.8]
 \foreach \i in {1,...,9} {
    \node[circle,fill, inner sep=2pt] (\i) at (\i*40:2) {};} 
    \draw[ultra thick,darkblue] (1) -- (2);
    \draw[ultra thick,darkblue] (1) -- (4);
    \draw[ultra thick,darkblue] (2) -- (3);
    \draw[ultra thick,darkblue] (2) -- (5);
    \draw[ultra thick,darkblue] (3) -- (4);
    \draw[ultra thick,darkblue] (3) -- (6);
    \draw[ultra thick,darkblue] (4) -- (5);
    \draw[ultra thick,darkblue] (4) -- (7);
    \draw[ultra thick,darkblue] (5) -- (6);
    \draw[ultra thick,darkblue] (5) -- (8);
    \draw[ultra thick,darkblue] (6) -- (7);
    \draw[ultra thick,darkblue] (6) -- (9);
    \draw[ultra thick,darkblue] (7) -- (8);
    \draw[ultra thick,darkblue] (7) -- (1);
    \draw[ultra thick,darkblue] (8) -- (9);
    \draw[ultra thick,darkblue] (8) -- (2);
    \draw[ultra thick,darkblue] (9) -- (1);
    \draw[ultra thick,darkblue] (9) -- (3);

    \draw (-1.25,1.85) node {$0$};
    \draw (0.25,2.25) node {$1$};
    \draw (1.5,1.6) node {$2$};
    \draw (2.1,0.25) node {$3$};
    \draw (1.5,-1.6) node {$4$};
    \draw (0.25,-2.25) node {$5$};
    \draw (-1.2,-1.95) node {$6$};
    \draw (-2,-1) node {$7$};
    \draw (-2,0.95) node {$8$};
 \foreach \i in {1,...,9} {
    \node[circle,fill, inner sep=2pt,white] (\i) at (\i*40:2) {};
      \node[circle,fill, inner sep=1.6pt,] (\i) at (\i*40:2) {};
    }
    
\begin{scope}[xshift=6cm]
    \foreach \i in {1,...,9} {
    \node[circle,fill, inner sep=2pt] (\i) at (\i*40:2) {};} 
    \draw[ultra thick,maroon] (1) -- (2);
    \draw[ultra thick,maroon] (2) -- (3);
    \draw[ultra thick,maroon] (3) -- (4);
    \draw[ultra thick,maroon] (4) -- (5);
    \draw[ultra thick,maroon] (5) -- (6);
    \draw[ultra thick,maroon] (6) -- (7);
    \draw[ultra thick,maroon] (7) -- (8);
    \draw[ultra thick,maroon] (8) -- (9);
    \draw[ultra thick,maroon] (9) -- (1);
    
 \draw (-1.25,1.85) node {$0$};
    \draw (0.25,2.25) node {$4$};
    \draw (1.5,1.6) node {$8$};
    \draw (2.1,0.25) node {$3$};
    \draw (1.5,-1.6) node {$7$};
    \draw (0.25,-2.25) node {$2$};
    \draw (-1.2,-1.95) node {$6$};
    \draw (-2,-1) node {$1$};
    \draw (-2,0.95) node {$5$};
     \foreach \i in {1,...,9} {
        \node[circle,fill, inner sep=2pt,white] (\i) at (\i*40:2) {};
      \node[circle,fill, inner sep=1.6pt,] (\i) at (\i*40:2) {};
    }
\end{scope}
    
\end{tikzpicture}
    \caption{The left blue graph is $G$ and the right red graph is $H$}
    \label{ex2}
\end{figure}
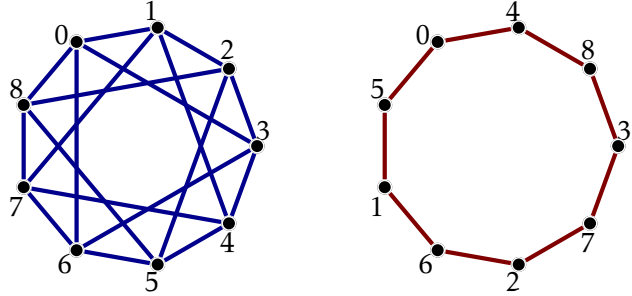
\end{exa}

\noindent In the following theorem, we generalize~\cref{example 10}, and we show that there is an infinite family of complete graphs that are MPF graphs via the product of the adjacency matrix of two regular graphs.

\begin{thm}\label{thm 12}
   A complete graph $ K_{4n+1} $ can be expressed as $\cay(\Z_{4n+1};S)_{g}\ast{}_{h}\cay(\Z_{4n+1}; R)$, where $S=\pm\{2k-1\mid k=0,\ldots,n-1\}$ and $R=\{2n,2n+1\}$.
\end{thm}

\begin{proof}
     Let $ G=\cay(\Z_{4n+1};S) $ be a simple graph with $ g\colon V(G) \to \{0,1,\ldots,4n\} $ as its labeling, and $ E(G) = \sqcup_{k=0}^{n-1} E_k $, where $ E_k = \{i \,\blueedge\, (i+2k-1) \pmod{4n+1} \mid i\in\{0,\ldots,4n\}\} $. 
     It is clear that $ G $ is $ 2n $-regular. 
     Also, let $ C_{4n+1}=\cay(\Z_{4n+1}; R) $ be a cycle with $ h\colon V(C_{4n+1}) \to \{0,1,\ldots,4n\} $, and $ E(C_{4n+1}) = \{i\,\rededge\, (i+2n) \pmod{4n+1}\mid i\in\{0,\ldots,4n\}\} $. 
     It is not hard to verify that $ C_{4n+1} $ is a cycle. 
     By the definition of $ C_{4n+1} $ and $ G $, they are loopless without multiple edges, so they are simple graphs. We want to show that $ C_{4n+1}{}_{g}\ast_{h} G $ is simple. It is clear that $ E(C_{4n+1})\cap E(G) = \emptyset $. Therefore, $ C_{4n+1}\oplus G $ is simple. So, by~\cref{lem 3}, we need to prove that $ C_{4n+1}\oplus G $ satisfies the diamond condition. Let $ E(G) $ be coloured by blue and $ 
    E(C_{4n+1}) $ be coloured by red.~(Note that all the vertices~(integers) are modulo $ 4n+1 $.)  Suppose $ i \in V(C_{4n+1}\oplus G) $, where $0\leq i\leq 4n-1$. 
    Then, there is a blue edge from $ i $ to $ i+2k-1 $ for each $  
k \in \{0,1,\ldots,n-1\} $ in $ C_{4n+1}\oplus G $. Also, we know by the definition of $ C_{4n+1} $, there are two red edges from $ i $ to $ i+2n $ and $ i-2n $ in $ C_{4n+1}\oplus G $. Clearly,
\begin{align*}
   i~\blueedge~(i+2k-1)~\rededge~(i+2k-1+2n)~\blueedge~(i+2n)~\rededge~i     
\end{align*}
is a 2-edge coloured diamond. We need to show that this is the only diamond between $ i $ and $ i+2k-1+2n $. Suppose there is another diamond that contains $ i $ and $ i+2k-1+2n $. There are three cases:\\\\
\textbf{Case1.} Assume there is a blue edge from $ i+2n $ to $ i+2k'-1+2n = i+2k-1+2n $ for some $ k' \in \{0,1,\ldots,n-1\} $, which implies that $ k = k' $.\\\\
\textbf{Case 2.} Assume that there is a blue edge from $ i-2n $ to $ i-2n+2k'-1 $, such that $ i-2n+2k'-1 = i+2k-1+2n $, which implies that $ 2k+2k'\equiv 1 \pmod{4n+1} $. This is a contradiction since $ 2k+2k'\leq 4n $.\\\\
\textbf{Case 3.} We know that there is a red edge from $ i+2k-1+2n $ to $ i+2k-2 $, and there is a blue edge from $ i+2k-2 $ to $ i+2k+2k'-3 $. If $ i = i+2k+2k'-3 $, then there are two 2-edge coloured diamond between $ i $ and $ i+2k-1+2n $. However, $ i = i+2k+2k'-3 $ implies that $ 2k+2k'-3 \equiv 0 \pmod{4n+1} $. This is a contradiction since $ 2k+2k' \leq 4n $.

\noindent Therefore, there is a 2-edge coloured between any two vertices of $ C_{4n+1}\oplus G $, it is satisfying the diamond condition in~\cref{lem 3}. Hence, $ K_{4n+1} = G{}_{g}\ast_{h}C_{4n+1}  $.    
\end{proof}

\completegraphs*

\begin{proof}
    It follows from \cref{odd-degree-vertex}, \cref{K_4n+3} and \cref{thm 12}. 
\end{proof}

\noindent The following example and theorem show that $ K_{4n+1} $ is not a unique MPF graph.

\begin{exa}
    The complete graph $ K_9 $ is an MPF graph. Let $ G = \cay(\mathbb{Z}_9;\{\pm 4,\pm 5\}) $ be a graph with $ g\colon V(G) \to \{0,1,\ldots,8\} $, and $ E(G) = \sqcup_{k=0}^1 E_k $, where $ E_k = \{i \,\blueedge \,(i+2+4k) \pmod{9}\mid i\in \{0,\ldots,8\}\} $, and $ H = \cay(\mathbb{Z}_9;\{\pm 2\}) $ be a graph with  $ h\colon V(H) \to \{0,1,\ldots,8\} $, and $ E(H) = \{i\,\rededge\, (i+2) \pmod{9}\mid i\in \{0,\ldots,8\}\} $. 
    It is not hard to verify that $ H $ is a $ 
C_9 $ cycle. Then we have the following two adjacency matrices:
\begin{align*}
    A(G) = \begin{bmatrix} 
	0 & 0 & 0 & 1 & 1 & 1 & 1 & 0 & 0 \\
	0 & 0 & 0 & 0 & 1 & 1 & 1 & 1 & 0 \\
    0 & 0 & 0 & 0 & 0 & 1 & 1 & 1 & 1 \\
	1 & 0 & 0 & 0 & 0 & 0 & 1 & 1 & 1 \\
 1 & 1 & 0 & 0 & 0 & 0 & 0 & 1 & 1 \\
 1 & 1 & 1 & 0 & 0 & 0 & 0 & 0 & 1 \\
  1 & 1 & 1 & 1 & 0 & 0 & 0 & 0 & 0 \\
  0 & 1 & 1 & 1 & 1 & 0 & 0 & 0 & 0 \\
  0 & 0 & 1 & 1 & 1 & 1 & 0 & 0 & 0 \\
	\end{bmatrix},~A(H) = \begin{bmatrix} 
	0 & 0 & 1 & 0 & 0 & 0 & 0 & 1 & 0 \\
	0 & 0 & 0 & 1 & 0 & 0 & 0 & 0 & 1 \\
    1 & 0 & 0 & 0 & 1 & 0 & 0 & 0 & 0 \\
	0 & 1 & 0 & 0 & 0 & 1 & 0 & 0 & 0 \\
 0 & 0 & 1 & 0 & 0 & 0 & 1 & 0 & 0 \\
 0 & 0 & 0 & 1 & 0 & 0 & 0 & 1 & 0 \\
  0 & 0 & 0 & 0 & 1 & 0 & 0 & 0 & 1 \\
  1 & 0 & 0 & 0 & 0 & 1 & 0 & 0 & 0 \\
  0 & 1 & 0 & 0 & 0 & 0 & 1 & 0 & 0 \\
	\end{bmatrix},\\
\end{align*}
which implies that $A(G)A(H) = J_9-I_9$, and this is the adjacency matrix of $ K_9 $.
\end{exa}

\begin{figure}[H]
    \centering
\begin{tikzpicture}[scale=0.8]
 \foreach \i in {1,...,9} {
    \node[circle,fill, inner sep=2pt] (\i) at (\i*40:2) {};} 
    \draw[ultra thick,darkblue] (1) -- (4);
    \draw[ultra thick,darkblue] (1) -- (5);
    \draw[ultra thick,darkblue] (1) -- (6);
    \draw[ultra thick,darkblue] (1) -- (7);
    \draw[ultra thick,darkblue] (2) -- (5);
    \draw[ultra thick,darkblue] (2) -- (6);
    \draw[ultra thick,darkblue] (2) -- (7);
     \draw[ultra thick,darkblue] (2) -- (8);
      \draw[ultra thick,darkblue] (3) -- (6);
      \draw[ultra thick,darkblue] (3) -- (7);
      \draw[ultra thick,darkblue] (3) -- (8);
      \draw[ultra thick,darkblue] (3) -- (9);
      \draw[ultra thick,darkblue] (4) -- (7);
      \draw[ultra thick,darkblue] (4) -- (8);
      \draw[ultra thick,darkblue] (4) -- (9);
      \draw[ultra thick,darkblue] (4) -- (1);

      \draw[ultra thick,darkblue] (5) -- (8);
      \draw[ultra thick,darkblue] (5) -- (9);
      \draw[ultra thick,darkblue] (6) -- (9);
    
    \draw (-1.25,1.85) node {$0$};
    \draw (0.25,2.25) node {$1$};
    \draw (1.5,1.6) node {$2$};
    \draw (2.1,0.25) node {$3$};
    \draw (1.5,-1.6) node {$4$};
    \draw (0.25,-2.25) node {$5$};
    \draw (-1.2,-1.95) node {$6$};
    \draw (-2,-1) node {$7$};
    \draw (-2,0.95) node {$8$};
 \foreach \i in {1,...,9} {
    \node[circle,fill, inner sep=2pt,white] (\i) at (\i*40:2) {};
      \node[circle,fill, inner sep=1.6pt,] (\i) at (\i*40:2) {};
    }
    
\begin{scope}[xshift=6cm]
    \foreach \i in {1,...,9} {
    \node[circle,fill, inner sep=2pt] (\i) at (\i*40:2) {};} 
    \draw[ultra thick,maroon] (1) -- (2);
    \draw[ultra thick,maroon] (2) -- (3);
    \draw[ultra thick,maroon] (3) -- (4);
    \draw[ultra thick,maroon] (4) -- (5);
    \draw[ultra thick,maroon] (5) -- (6);
    \draw[ultra thick,maroon] (6) -- (7);
    \draw[ultra thick,maroon] (7) -- (8);
    \draw[ultra thick,maroon] (8) -- (9);
    \draw[ultra thick,maroon] (9) -- (1);
    
 \draw (-1.25,1.85) node {$0$};
    \draw (0.25,2.25) node {$2$};
    \draw (1.5,1.6) node {$4$};
    \draw (2.1,0.25) node {$6$};
    \draw (1.5,-1.6) node {$8$};
    \draw (0.25,-2.25) node {$1$};
    \draw (-1.2,-1.95) node {$3$};
    \draw (-2,-1) node {$5$};
    \draw (-2,0.95) node {$7$};
     \foreach \i in {1,...,9} {
        \node[circle,fill, inner sep=2pt,white] (\i) at (\i*40:2) {};
      \node[circle,fill, inner sep=1.6pt,] (\i) at (\i*40:2) {};
    }
\end{scope}
    
\end{tikzpicture}
    \caption{The left blue graph is $G$ and the right red graph is $H$}
\end{figure}
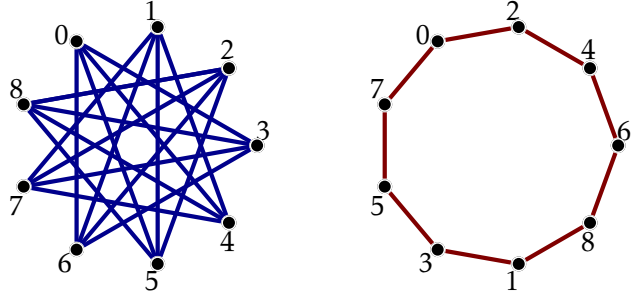

\begin{thm}
Let $n$ be even. Then  $K_{4n+1}$ can be expressed as $\cay(\mathbb Z_{4n+1};S){}_{g}\ast{}_{h}  \cay(\mathbb Z_{4n+1};R)$, where $S=\{\pm 2n,\pm(2n-1) \}$ and $R=\{\pm(2+4k)\mid k=0,\ldots,\frac{n}{2}-1\}$.
\end{thm}

\begin{proof}
Let $G=\cay(\mathbb Z_{4n+1};\pm 2,\pm 6,\ldots, \pm(2+2n))$. 
Notice that $V(G)=\{0,\ldots, 4n\}$, and we color all edges of $G$ in blue. The edges $E(G)$ are composed of $\cup_{k=0}^{(n-2)/2}G_k$, where 
\[
G_k=\{j \blueedge j+2+4k, j \blueedge j-2-4k \mid j=0,\ldots 4n\}, \pmod{4n+1}
\]

\noindent Let $H=\cay(\mathbb Z_{4n+1};\pm 2n,\pm(2n-1))$.
We note that $V(H)=\{0,\ldots, 4n\}$, and all edges of $H$ are colored in red, where
\[
E(H)=\{i\rededge i+2n, i \rededge i-2n \mid i=0,\ldots 4n\} \pmod{4n+1}
\]

\noindent Firstly, we show that $G\oplus H$ is simple, which means there are no multiple edges in $G\oplus H$.
Both $G$ and $H$ are simple graphs. 
Assume, to the contrary, that $u \rededge v$ and $u \blueedge v$.
Let $u=i$, where $i\in\{0,\ldots,4n\}$. 
We consider the following cases:

\noindent \textbf{Case 1.} Let $v=i+2n$. 
Assume $v=i+2+4\ell$ for $\ell \in \{0,\ldots,\frac{n}{2}-1\}$.
We have $2n\equiv 2+4\ell \pmod{4n+1}$, implying $n= 1+2\ell$. This leads to a contradiction, as $n$ is even.
Next, assume $v=i-2-4\ell$. 
Similarly, $-n\equiv 1+2\ell \pmod{4n+1}$, meaning $3n=2\ell$, which again leads to a contradiction as $3n=2\ell\leq n-2$.

\noindent \textbf{Case 2.} Let $v=i-2n$.
Assume $v=i+2+4\ell$ for $\ell \in \{0,\ldots,\frac{n}{2}-1\}$.
We have $-2n\equiv 2+4\ell \pmod{4n+1}$, implying $3n= 2\ell$, which leads to a contradiction as $3n=2\ell\leq n$. 
Next, assume $v=i-2-4\ell$. 
Similarly, $-n\equiv -1-2\ell \pmod{4n+1}$, meaning $n=1+2\ell$, which again leads to a contradiction as $n$ is even.

\noindent Next, we aim to demonstrate that $G\oplus H$ satisfies the diamond condition. 
Suppose we have $u \rededge v\blueedge w \rededge z \blueedge u$.
Let's analyze the following cases:
\begin{case} 
Let $u=i$, $v=i+2n$, and $w=i+2n+2+4k$, which implies $i \rededge i+2n\blueedge i+2n+2+4k$ for some $ k\in \{0,1,\ldots,n/2-1\}$. We can break this down into $4$ subcases:\\\\
\noindent \textbf{Subcase 1.1.} Clearly, $(i+2n+2+4k)\rededge(i+4n+2+4k)$. 
Assume $(i+4n+2+4k) \blueedge(i+4n+2+4k+2+4k')$ for some $ k'\in\{0,1,\ldots,n/2-1\}$ such that $ i+4n+2+4k+2+4k' \equiv i \pmod{4n+1} $. This implies $4(k+k')+3 \equiv 0 \pmod{4n+1}$. 
Observing that $4n+1\leq 4(k+k')+3$ and considering the maximum value of $k+k'\leq n-2$, we find that $4n+1\leq 4(k+k')+3\leq 4(n-2)+3=4n-5$, which leads to a contradiction.\\

\noindent We may now assume $ (i+4n+2+4k)\blueedge(i+4n+2+4k-2-4k'') $ for some $ k''\in\{0,1,\ldots,n/2-1\} $, such that $ i+4n+2+4k-2-4k'' \equiv i \pmod{4n+1} $. This implies that $ n+k-k'' \equiv 0 \pmod{4n+1} $. Note that $ 4n+1 \leq n+k-k'' $, considering the maximum value of $ k-k'' \leq n/2-1 $, we see that $ 4n+1 \leq n+k-k'' \leq (3n)/2-1 $, which is a contradiction. 
\\\\
\noindent \textbf{Subcase 1.2.} Similarly, $(i+2n+2+4n)\rededge(i+2+4k)$. Assume $(i+2+4k)\blueedge(i+2+4k+2+4k')$ for some $ k'\in\{0,1,\ldots,n/2-1\}$ such that $ i+2+4k+2+4k' \equiv i \pmod{4n+1} $. This implies $4(1+k+k') \equiv 0 \pmod{4n+1}$. Observing that $4n+1\leq 4(k+k'+1)$ and considering the maximum value of $k+k'\leq n-2$, we find that $4n+1\leq 4(k+k'+1)\leq 4(n-2+1)=4n-4$, which leads to a contradiction.\\

\noindent Now we may assume that $ (i+2+4k)\blueedge(i+2+4k-2-4k'') $, for some $ k'' \in\{0,1,\ldots,n/2-1\} $, such that $ i+2+4k-2-4k'' \equiv i \pmod{4n+1} $. This implies that $ k-k''\equiv 0 \pmod{4n+1} $. Hence, $ k=k'' $, and we have the following colored diamond:
$$ i\rededge (i+2n)\blueedge i+2n+2+4k \rededge i+2+4k \blueedge i+2+4k-2-4k'' = i. $$\\
\noindent \textbf{Subcase 1.3.} Likewise, $(i+2n+2+4k)\rededge(i+4k)$. Assume $(i+4k)\blueedge(i+4k+2+4k')$, for some $ k'\in\{0,1,\ldots,n/2-1\}$ such that $(i+4k+2+4k') \equiv i \pmod{4n+1}$. 
Observing that $4n+1\leq 4(k+k')+2$, considering the maximum value of $k+k'\leq n-2$, we find that $4n+1\leq 4(k+k')+2\leq 4(n-2)+2=4n-6$, which leads to a contradiction.\\

\noindent Now we may assume that $ (i+4k)\blueedge (i+4k-2-4k'') $ for some $ k'' \in\{0,1,\ldots,n/2-1\} $, such that $ i+4k-2-4k'' \equiv i \pmod{4n+1} $. This implies that $ 2k-1-2k'' \equiv 0 \pmod{4n+1} $. Note that $ 4n+1 \leq 2(k-k'')-1 \leq n-3 $, which is a contradiction.
\\\\
\noindent \textbf{Subcase 1.4.} Lastly, $(i+2n+2+4k)\rededge(i+4k+3)$. Assume $(i+4k+3)\blueedge(i+4k+3+2+4k')$, for some $ k'\in\{0,1,\ldots,n/2-1\}$ such that $(i+4k+3+2+4k')\equiv i \pmod{4n+1}$. 
This implies $4(k+k'+1)+1\equiv 0 \pmod{4n+1}$, which implies that $ n+k+k'+1 \equiv 0 \pmod{4n+1} $. Hence, $ 4n+1 \leq n+(k+k')+1 \leq 2n-1 $, which leads to a contradiction.\\ 

\noindent Now we may assume that $ (i+4k+3) \blueedge (i+4k+3-2-4k'') $ for some $k'' \in\{0,1,\ldots,n/2-1\} $, such that $ i+4k+3-2-4k'' \equiv i \pmod{4n+1} $. This implies that $ 4(k-k'')+1 \equiv 0 \pmod{4n+1} $, hence, $ n+k''-k \equiv 0\pmod{4n+1} $, thus, $ 4n+1 \leq n+k''-k \leq (3n)/2-1 $, a contradiction.
\end{case}
\noindent The other cases can be analyzed analogously to the preceding one, following a parallel method of reasoning.
\end{proof}

\begin{conj}
    Let $ K_{4n+1} $ be a complete graph on $ 4n+1 $ vertices, where $ 4n = m_1\cdot m_2 $ and both $ 
m_1 $ and $ m_2 $ are even. Then there is a graphical pair $ ((G,g),(H,h)) $ without loops where $G$ and $H$ are $m_1$-regular and $m_2$-regular graphs, such that $ 
K_{4n+1} = G_g\ast{}_h H $.
\end{conj}

\subsection{Complete Multi-partite Graphs}
Another known family of graphs that we can discuss whether it is an MPF graph or not, is the complete $ m $-partite graph $ K_{n_1,\ldots,n_m} $.

\begin{prop}\label{K_mn-odd}
    If either $m$ or $n$ is an odd positive integer,  
    then $K_{m,n}$ is not MPF.
\end{prop}
\begin{proof}
Without loss of generality we assume that $n$ is odd.  
By \cref{anylabel} we can assume that  
$A=\begin{bmatrix}
        \bigzero_{m} & J_{m,n} \\
        J_{n,m} & \bigzero_n
    \end{bmatrix}$ 
is the adjacency matrix of $K_{m,n}$, and assume that there exist graphs $G$, and $ H $ with adjacency matrices $A(G)$, and $A(H)$, respectively such that $A=A(G)A(H)$. 
Let $V_1=\{1,\ldots,m\}$ and $V_2=\{m+1,\ldots,m+n\}$, and let $x\in V_1$ with $\ell=|N_{G}(x)\cap V_1|$ and $\ell'=|N_{G}(x)\cap V_2|$. By suitably permuting the vertices $ V_1 $ and $ V_2 $ in $ G $, we may assume $ A(G)^{\ast} = P^{-1}A(G)P $ has the following form~($ P $ is a permutation matrix):
$$
A(G)^{\ast} = \left[\begin{array}{c|c|c|c|c}
0 & {\bf 1}_{\ell}^\top & {\bf 0}_{m-\ell-1}^\top & {\bf 1}_{\ell'}^\top & {\bf 0}_{n-\ell'}^\top \\
\hline
{\bf 1}_{\ell} & * & * & * & * \\
\hline
{\bf 0}_{m-\ell-1} & * & * & * & * \\
\hline
{\bf 1}_{\ell'} & * & * & * & * \\
\hline
{\bf 0}_{n-\ell'} & * & * & * & * \\
\end{array}\right]. 
$$ 
Using the same permutation matrix $ P $, we relabel the graph $ H $ to get $ A(H)^{\ast} = P^{-1}A(H)P $. We set 
$$
A(H)^{\ast}=\left[\begin{array}{c|c|c|c|c}
0 & u_1^\top & u_2^\top & u_3^\top & u_4^\top \\
\hline
* & * & * & * & * \\
\hline
* & * & * & * & * \\
\hline
* & * & * & * & * \\
\hline
* & * & * & * & * \\
\end{array}\right].
$$
Since the permutation $ P $ permutes the vertices $ V_1 $ among themselves and similarly, it permutes the vertices $ V_2 $ among themselves, we have $A= P^{-1}AP =A(G)^{\ast}A(H)^{\ast}$. 
Comparing the $(2,i)$-entry of $A=A(G)^{\ast}A(H)^{\ast}$ for $i\in\{2,\ldots,m\}$, we have 
$u_1^\top=u_2^\top={\bf 0}^\top$ provided that $\ell>0$.
Similarly we have $u_3^\top=u_4^\top={\bf 0}^\top$ provided that $\ell'>0$.

If both $\ell$ and $\ell'$ are positive, then the first row of $A(H)^{\ast}$ is the zero vector, which contradicts $A=A(G)^{\ast}A(H)^{\ast}$. 
Therefore,  for each $x\in V_1$, either $\ell=0$ or $\ell'=0$. 

Take $x\coloneqq 1\in V_1$. 
\begin{enumerate}
    \item 
First, we assume $\ell'=0$. 
Then we may assume that $A(G)$ has the following form:
$$
A(G)=\left[\begin{array}{c|c|c|c}
0 & {\bf 1}_{\ell}^\top & {\bf 0}_{m-\ell-1}^\top & {\bf 0}_{n}^\top  \\
\hline
{\bf 1}_{\ell} & * & * & *  \\
\hline
{\bf 0}_{m-\ell-1} & * & * & *  \\
\hline
{\bf 0}_{n} & * & * & *  \\
\end{array}\right]. 
$$
By $A=A(G)A(H)$, assume that $A(H)$ has the following form:  
$$
A(H)=\left[\begin{array}{c|c|c|c|c}
0 & {\bf 0}_{\ell}^\top & {\bf 0}_{m-\ell-1}^\top & {\bf 1}_{\ell''}^\top & {\bf 0}_{n-\ell''}^\top \\
\hline
{\bf 0}_{\ell} & * & * & * & * \\
\hline
{\bf 0}_{m-\ell-1} & * & * & * & * \\
\hline
{\bf 1}_{\ell''} & * & * & * & * \\
\hline
{\bf 0}_{n-\ell''} & * & * & * & * \\
\end{array}\right], 
$$
where $\ell''\coloneqq |N_{H}(1)\cap V_2|$. 
Then according to this block form, $A(G)$ can be partitioned into the following form: 

$$
A(G)=\left[\begin{array}{c|c|c|c|c}
0 & {\bf 1}_{\ell}^\top & {\bf 0}_{m-\ell-1}^\top & {\bf 0}_{\ell''}^\top & {\bf 0}_{n-\ell''}^\top \\
\hline
{\bf 1}_{\ell} & * & * & * & * \\
\hline
{\bf 0}_{m-\ell-1} & * & * & * & * \\
\hline
{\bf 0}_{\ell''} & * & * & B & * \\
\hline
{\bf 0}_{n-\ell''} & * & * & * & * \\
\end{array}\right], 
$$
for some adjacency matrix $B$ of size $\ell''$. 
Then by calculating $A=A(G)A(H)=A(H)A(G)$, we have 
\[
B{\bf 1}={\bf 1}, 
\]
which implies $B$ is the adjacency matrix of a perfect matching. 
Thus $\ell''$ must be even. 
However, $\ell''$ must be a divisor of $n$ by \cref{thm 4}, which contradicts the assumption that $n$ is odd.   
\item Next we assume $\ell=0$. 

Then we may assume that $A(G)$ has the following form:
$$
A(G)=\left[\begin{array}{c|c|c|c}
0 &  {\bf 0}_{m-1}^\top & {\bf 1}_{\ell'}^\top & {\bf 0}_{n-\ell'}^\top  \\
\hline
{\bf 0}_{m-1} & * & * & *  \\
\hline
{\bf 1}_{\ell'} & * & * & *  \\
\hline
{\bf 0}_{n-\ell'} & * & * & *  \\
\end{array}\right]. 
$$  
By $A=A(G)A(H)$, assume that $A(H)$ has the following form:  
$$
A(H)=\left[\begin{array}{c|c|c|c|c}
0 & {\bf 1}_{\ell''}^\top & {\bf 0}_{m-\ell''-1}^\top & {\bf 0}_{\ell'}^\top & {\bf 0}_{n-\ell'}^\top \\
\hline
{\bf 1}_{\ell''} & * & * & * & * \\
\hline
{\bf 0}_{m-\ell''-1} & * & * & * & * \\
\hline
{\bf 0}_{\ell'} & * & * & B & * \\
\hline
{\bf 0}_{n-\ell'} & * & * & * & * \\
\end{array}\right], 
$$
where $ \ell'' = |N_{H}(1)\cap V_1| $, for some adjacency matrix $B$ such that $B{\bf 1}={\bf 1}$.  

Then according to this block form, $A(G)$ can be partitioned into the following form: 
$$
A(G)=\left[\begin{array}{c|c|c|c|c}
0 & {\bf 0}_{\ell''}^\top & {\bf 0}_{m-\ell''-1}^\top & {\bf 1}_{\ell'}^\top & {\bf 0}_{n-\ell'}^\top \\
\hline
{\bf 0}_{\ell''} & * & * & * & * \\
\hline
{\bf 0}_{m-\ell''-1} & * & * & * & * \\
\hline
{\bf 1}_{\ell'} & * & * & * & * \\
\hline
{\bf 0}_{n-\ell'} & * & * & * & * \\
\end{array}\right],   
$$
and again by considering $A=A(H)A(G)$, we have 
$B{\bf 1}={\bf 1}$. 
Then $B$ is the adjacency matrix of a perfect matching, so $\ell'$ must be even. 
However $\ell' = |N_{H}(1)\cap V_2|$ is a divisor of $n$ by \cref{thm 4}. Since $n$ is odd, we have a contradiction. 
\qedhere
\end{enumerate}
\end{proof}

\begin{prop}\label{K_mn-even}
    The complete bipartite $ K_{2n,2m} $ is an MPF graph.
\end{prop}

\begin{proof}
    Let $A$ be the adjacency matrix of the graph $K_2$.
    Then define $D\coloneqq \mathsf{diag}(A,\ldots,A)$, where $\mathsf{diag}$ is the diagonal matrix, whose entries of the block matrix $A$. 
    Now define the following matrix
$$
    \left[\begin{array}{c|c}
D_{n\times n} &  \bigzero_{n\times m} \\
\hline
\bigzero_{m\times n} & D_{m\times m} \\
\end{array}\right].
$$
This is the adjacency matrix corresponding to $ n+m $ matchings, say graph $ H $ with the labeling $h$. 
Now consider the following labeling $g$ for $ K_{2n,2m} $:
$$
    A(K_{2n,2m}) = \left[\begin{array}{c|c}
\bigzero_{n\times n} &  J_{n\times m} \\
\hline
J_{m\times n} & \bigzero_{m\times m} \\
\end{array}\right].
$$
    It is not hard to see that $D_{n\times n}J_{n\times m}=J_{n\times m}$ and so $H{}_h\ast{}_g K_{2n,2m}{}=K_{2n,2m}$.
\end{proof}

\completebipartitegraphs*

\begin{proof}
    It follows from \cref{K_mn-even} and \cref{K_mn-odd}.
\end{proof}

\begin{lem}
    Let $ K_{m_1,m_2,\ldots,m_r} $ be a complete multi partite graph with $ \ell $ number of partitions containing odd number of vertices. If $ \ell \equiv 2 \pmod{4} $ or $ \ell \equiv 3 \pmod{4} $, then $ K_{m_1,m_2,\ldots,m_r} $ is not an MPF graph.
\end{lem}

\begin{proof}
    Assume $ \ell \equiv 2 \pmod{ 4} $, and let $ m_{i_1}, m_{i_2},\ldots,m_{i_{\ell}} $ be the $ \ell  $ odd partitions, where $ i_{1},\ldots,i_{\ell} \in \{1,\ldots,r\} $. Also, let $ E = E_{e}+E_{o} $ be the number of edges of $ K_{m_1,\ldots,m_r} $, where $  E_{e} $ is the number of edges between even partitions with themselves and the number of even partitions with odd partitions, and $ E_o $ is the number of edges between $ \ell $ odd partitions. Now we count the number of edges of $ E_o $, so we have
    \begin{align*}
        E_o = m_{i_0}\sum_{k=2}^{\ell}m_{i_k}+m_{i_1}\sum_{k=3}^{\ell}m_{i_k}+\ldots+m_{i_{\ell-1}}m_{i_\ell}.
    \end{align*}
    It is not hard to see that the parity of $ E_o $ is odd. Hence, $ E $ is odd, and by \cref{odd-edges}, $ K_{m_1,\ldots,m_r} $ is not an MPF graph.

    \noindent With a similar argument, it can be shown that when $ \ell \equiv 3 \pmod{4} $, then the number of edges of $ K_{m_1,\ldots,m_r} $ is odd, and \cref{odd-edges} applies.
\end{proof}

\begin{thm}
Complete $m$-partite graph $ K_{2n,2n,\ldots,2n} $ is an MPF graph for $ m,n\geq 2 $.
\end{thm}

\begin{proof}
    Let $ G $ be the graph corresponding to the following matrix:
    $$
    \left[\begin{array}{c|c|c|c}
\bigzero_n~~\bigzero_n &  I_n~~I_n & \cdots & I_n~~I_n \\
\bigzero_n~~\bigzero_n & I_n~~I_n & \cdots &I_n~~I_n \\
\hline
I_n~~I_n &\bigzero_n~~\bigzero_n &\cdots &I_n~~I_n \\
I_n~~I_n & \bigzero_n~~\bigzero_n & \cdots & I_n~~I_n \\
\hline
\vdots & \vdots & \ddots & \vdots \\
\hline
I_n~~I_n & I_n~~I_n & \cdots & \bigzero_n~~\bigzero_n \\
I_n~~I_n & I_n~~I_n &\cdots & \bigzero_n~~\bigzero_n \\
\end{array}\right].
$$
Since it is a (0,1)-matrix that is also symmetric, with 0 in its diagonal entries, then $ G $ is simple. Also, it is clear that $ G $ is $ 2(m-1) $-regular. So, this matrix is $ A(G) $.

\noindent Also, let $ H $ be $ n $-regular disconnected graph of $ m $ disjoint union of $ K_{n,n} $. We let $ H = \sqcup_{i=0}^{m-1}K_{n,n}^i $ and $ h\colon V(H) \to \{0,1,\ldots,2mn-1\}$ such that the first part of $ K_{n,n}^i $ in $ H $ contains the set of vertices $ \{in+1,in+2,\ldots in+n\} $ and the second part of $ K_{n,n}^i $ contains the vertices $ \{(i+1)n+1,(i+1)n+2,\ldots,(i+1)n+n\} $ for $ i \in \{0,2,4,\ldots,2m-2\} $. Hence, $A(H)_{2mn\times 2mn} =$
    $$
     \left[\begin{array}{c|c|c|c}
\bigzero_n~~J_n &  \bigzero_n~~\bigzero_n & \cdots & \bigzero_n~~\bigzero_n \\
J_n~~\bigzero_n & \bigzero_n~~\bigzero_n & \cdots &\bigzero_n~~\bigzero_n \\
\hline
\bigzero_n~~\bigzero_n &\bigzero_n~~J_n &\cdots &\bigzero_n~~\bigzero_n \\
\bigzero_n~~\bigzero_n & J_n~~\bigzero_n & \cdots & \bigzero_n~~\bigzero_n \\
\hline
\vdots & \vdots & \ddots & \vdots \\
\hline
\bigzero_n~~\bigzero_n & \bigzero_n~~\bigzero_n & \cdots & \bigzero_n~~J_n \\
\bigzero_n~~\bigzero_n & \bigzero_n~~\bigzero_n &\cdots & J_n~~\bigzero_n \\
\end{array}\right],
$$
    Therefore, $A(G)A(H)=A(H)A(G) =$
$$
     \left[\begin{array}{c|c|c|c}
\bigzero_n~~\bigzero_n &  J_n~~J_n & \cdots & J_n~~J_n \\
\bigzero_n~~\bigzero_n & J_n~~J_n & \cdots &J_n~~J_n \\
\hline
J_n~~J_n &\bigzero_n~~\bigzero_n &\cdots &J_n~~J_n \\
J_n~~J_n & \bigzero_n~~\bigzero_n & \cdots & J_n~~J_n \\
\hline
\vdots & \vdots & \ddots & \vdots \\
\hline
J_n~~J_n & J_n~~J_n & \cdots & \bigzero_n~~\bigzero_n \\
J_n~~J_n & J_n~~J_n &\cdots & \bigzero_n~~\bigzero_n \\
\end{array}\right],
$$
which is clearly the adjacency matrix of $ K_{2n,2n,\ldots,2n} $. This implies that $ K_{2n,2n,\ldots,2n} $ is an MPF graph.    
\end{proof}

\subsection{Other Graph Classes}

\begin{lem}\label{lem:ei}
Assume that $A=A_1A_2$ holds for symmetric $(0,1)$-matrices $A,A_1$, and $A_2$ of size $ n\times n $. 
If the $j$-th row vector of $A_1$ is $e_i^t$, then  the $i$-th row vector of $A_2$ coincides with the $j$-th row vector of $A$. 
\end{lem}

\begin{proof}
    Let $ \textnormal{Col}_k(A_2) $ be the $ 
k $-th column of $ A_2 $, and $ \textnormal{Row}_j(A_2) $ be the $ j $-th row of $ A_2 $. Since $ A_1A_2 = A $, and $ j $-th row of $ A_1 $ is $ e_i^t $, then the $ j $-th row of $ A $ is
    \begin{align*}
        A_{(j,k)} = \sum_{k=1}^n e_i^t\cdot \textnormal{Col}_k(A_2) = \textnormal{Row}_j(A_2),
    \end{align*}
    as desired.
\end{proof}

\begin{prop}\label{frined}
    The friendship graph $F_n$ is not MPF.
\end{prop}
\begin{proof}
Let $A=\begin{bmatrix}
0 & {\bm 1}^\top\\
{\bm 1} & I_{n}\otimes (J_2-I_2)  
\end{bmatrix}$.  
Assume that there exist adjacency matrices $A_1, A_2$ such that $A=A_1A_2$. 
Since the valency of the vertex $2$ in $F_n$ is two, the degree of the vertex $2$ in $A_1, A_2$ are one or two by Lemma~3.5. 
By changing the role of $A_1$ and $A_2$ if necessary, We may assume that the second row of $A_1$ is $e_i$ for some $i\in\{1,3,4,\ldots,n\}$. 
By Lemma~\ref{lem:ei}, the $i$-th row of $A_2$ is $(1010\cdots0)$. 
Then by calculating the $(i,1)$-entry of $A_2A_1=A$, we have $(A_1)_{3,1}=1$ and thus $A_{1,3}=1$. 
Thus $A_{i,3}=(A_2A_1)_{i,3}=1$, which shows that $i=2$. Therefore, it turns out that the $(2,2)$-entry of $A_1$ equals $1$, a contradiction.
\end{proof}
\begin{defn}
Let $d$ be a positive integer greater than $1$. 
The binary Hamming association scheme is  $H(d,2)=(X,\{R_i\}_{i=0}^d)$ with $X=\{0,1\}^d$ and 
\begin{align*}
R_i=\{(x,y)\in X\times X \mid d(x,y)=i\}, 
\end{align*}
where $d(x,y)$ is the Hamming distance between $x$ and $y$.
\end{defn}

\noindent The graph $(X,\tilde{R}_1)$, where $\tilde{R}_\ell=\{\{x,y\} \mid (x,y\in R_\ell)\}$, is known as the hypercube $Q_d$. For any $\ell\in\{1,\ldots,d\}$, let $A_\ell$ be the adjacency matrix of the graph $(X,\tilde{R}_\ell)$, and set $A_0=I$. 

\begin{thm}\label{thm:cube}
    The hypercube $Q_d$ is an MPF graph. 
\end{thm}
\begin{proof}
Since $p_{d-1,d}^1=1$ and $p_{d-1,d}^k=0$ for $k\in\{0,2,\ldots,d\}$, 
$
   A_{d-1}A_d=\sum_{k=0}^d p_{d-1,d}^k A_k=A_1
$. 
\end{proof}

\noindent Up to this point, we have noted that many graph classes do not possess matrix product factorization. This naturally leads to the question: Is there a specific operation that can be applied to a given graph $G$ to produce a new graph $\widetilde{G}$ for which $\widetilde{G}$ possesses a matrix product factorization? In the following definition, we introduce an operation for this purpose.

\begin{defn}\label{Gtilde}
    Let $G$ be a simple graph on $n$ vertices. 
    Then $\widetilde{G} $ is a graph with the vertices $\{(i,v_j)\mid  v_j\in V(G),i=1,2\}$ and the edges:
    $$(1,v_i)\sim (2,v_j) \Longleftrightarrow v_i\sim v_j$$
\end{defn}

\begin{lem}\label{Xtilde}
    Let  $ X $ be a graph.
    Then $ \widetilde{X} $ is an MPF graph.
\end{lem}

\begin{proof}
Let  $ f\colon V(X)\to \{1,2,\ldots,n\} $ be a labeling of $X$. 
Then consider a labeling of  $\widetilde{X}$ such that $ A(\widetilde{X}) = A(K_2)\otimes A(X) $. Let $ G $ be a disjoint union of $2$ copies of $X$ and then set the labeling $g$ of $G$ such that $ A(G) = I_2\otimes A(X) $. Now suppose $ 
H $ is a disjoint union of $n$ copies of $K_2$.
We set the labeling $h$ of $H$ such that $ A(H) = A(K_2)\otimes I_n $. Therefore, it is not hard to see that $ A(G)A(H) = A(\widetilde{X}) $, which implies that $ G_g\ast_h H = \widetilde{X} $.
\end{proof}

\begin{cor}\label{C-even}
    Every cycle $ C_{2n} $ is an MPF graph.
\end{cor}

\begin{proof}
    It is not hard to see that $ \widetilde{C_n} = C_{2n} $. Then the argument of \cref{Xtilde} works.
\end{proof}

\begin{thm}
\label{cycle} 
    The cycle graph $ C_n $ is an MPF graph if and only if $ n $ is even.
\end{thm}

\begin{proof}
    It follows from \cref{C_odd} and \cref{C-even}.
\end{proof}

\begin{lem}\label{selfMPF}
Let $G$ be a simple graph on $2n$ vertices and let $X$ be the graph defined with $ J_2\otimes A(G) $ as its adjacency matrix. Then $ X $ is an MPF graph.
\end{lem}

\begin{proof}
    Let $g\colon V(G)\to\{1,\ldots,n\}$ be the labeling defined by $A(G)$.
    Then we define the graph $H$ as the disjoint union of $n$ copies of $K_2$.
    Next consider the labeling $h\colon V(H)\to \{1,\ldots,n\}$ such that $ 
A(H) = A(K_2)\otimes I_n $.
Now it is not hard to see that $G_g\ast _h H= X$, as desired.
\end{proof}

\noindent We already show that $ C_{2n} $ is an MPF graph in \cref{C-even}. In the following corollary, we show that $ C_{4n} $ is an MPF graph into $C_{4n}$ and a matching.

\begin{cor}
    Every cycle $C_{4n}$ of length $4n$ is an MPF graph into $C_{4n}$ and a matching.
\end{cor}
\begin{proof}
    Let $A$ be the adjacency matrix of the graph $K_2$.
    Then define $D\coloneqq \mathsf{diag}(A,\ldots,A)$, where $\mathsf{diag}$ is the diagonal matrix, whose entries of the block matrix $A$.
    Next, it is not hard to see that there is a labeling of $C_{4n}$ such that $ A(C_{4n}) = J_2 \otimes D $. Now it follows from \Cref{selfMPF} that $C_{4n}$ is an MPF graph $C_{4n}$ and a matching.
\end{proof}

\begin{defn}
    The \textit{web graph} $ W_{2n,2} $ is a graph consisting of $ 2 $ copies of the cycle graph $C_{2n}$, with corresponding vertices connected by spokes.
\end{defn}

\begin{figure}[H]
    \centering
\includegraphics[scale=1.1]{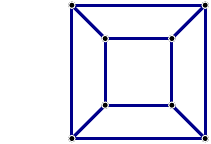}
    \caption{The web graph $ W_{4,2} $}
    \label{web-graph}
\end{figure}

\begin{cor}\label{webgraph}
    Every web graph $ W_{2n,2} $ is an MPF graph.
\end{cor}

\begin{proof}
    Let $ R $ be a simple graph on $ 
2n $ vertices with $ \rho\colon V(R) \to \{0,1,\ldots,2n-1\} $ as its labeling, and $  E(R) = E(R_1)\sqcup E(R_2) $, where $ 
E(R_1) = \{i~\blueedge~i+1 \pmod{2n}\mid i \in\{0,1,\ldots,2n-1\}\} $ and $ E(R_2) = \{i~\blueedge~i+n \pmod{2n} \mid i \in\{0,1,\ldots,2n-1\}\} $. It is clear that $ R $ is a $3$-regular graph. Now we show that $ \widetilde{R}$ is isomorphic to $W_{2n,2} $. It is also clear that $ \widetilde{R} $ is 3-regular. By \cref{Gtilde}, $ V(\widetilde{R}) = \{(i,j) \mid j = \rho(v), v \in V(R), j \in\{0,1,\ldots,2n-1\}, i\in\{1,2\}\} $. Now without loss of generality starting from vertex $ (1,0) $ and $ (1,n) $ in $ \widetilde{R} $ and using the definition of $ E(R_1) $, we have the following two cycles $ C_{2n} $ in $ \widetilde{R} $:
 \begin{align*}
C\colon~(1,0)~\blueedge~(2,1)~\blueedge~(1,2)~\blueedge\cdots\blueedge~(2,n-1)~\blueedge~(1,n)~\blueedge\cdots\blueedge~(2,2n-1)~\blueedge~(1,0),\\\\
C'\colon~(2,0)~\blueedge~(1,1)~\blueedge~(2,1)~\blueedge\cdots\blueedge~(1,n-1)~\blueedge~(2,n)~\blueedge\cdots\blueedge~(1,2n-1)~\blueedge~(2,0).
 \end{align*}
These two cycles are disjoint and contain all vertices in $ \widetilde{R} $. Now by rotating $ C' $, we have
 \begin{align*}
C\colon~(1,0)~\blueedge(2,1)~\blueedge~(1,2)~\blueedge\cdots\blueedge~(2,n-1)~\blueedge~(1,n)~\blueedge\cdots\blueedge~(2,2n-1)~\blueedge~(1,0),\\\\
C'\colon~(2,n)~\blueedge~(1,n+1)~\blueedge~(2,n+2)~\blueedge\cdots\blueedge~(1,2n-1)~\blueedge~(2,0)~\blueedge\cdots\blueedge~(1,n-1)~\blueedge~(2,n).
 \end{align*}
By using $ E(R_2) $, we know that there is an edge between $ (1,0) $ and $ (2,n) $ of $ C $ and $ C' $, similarly one can see that there is a perfect matching between $ C $ and $ C' $ in $ \widetilde{R} $, which is isomorphic to $ W_{2n,2} $. 
\end{proof}

\begin{defn}
The circulant graph $ Ci_{2n+1}(n-1,n) $ is a graph on $ 2n+1 $ vertices where $ V(Ci_{2n+1}) = \{0,1,\ldots,2n\} $ and $ E(Ci_{2n+1})=\{i \sim i+n-1 \pmod{2n+1} $ and $ 
i \sim i+n \pmod{2n+1} $ for each $ 
i \in V(Ci_{2n+1}) $\}.
\end{defn}

\begin{prop}\label{circ}
    The circulant graph $ Ci_{2n+1}(n-1,n)  $ is MPF for $ n\geq 2 $.
\end{prop}

\begin{proof}
    Let $ C_{2n+1} $ and $ C'_{2n+1} $ be cycles with $ g\colon V(C_{2n+1}) \to \{0,1,\ldots,2n\} $ and $ h\colon V(C_{2n+1}) \to \{0,1,\ldots,2n\}  $, respectively, such that $ E(C_{2n+1}) = \{i~\blueedge~(i+1) \pmod{2n+1}\mid i\in\{0,1,\ldots,2n\}\} $ and $ E(C'_{2n+1}) = \{i~\rededge~(i+n) \pmod{2n+1}\mid i\in\{0,1,\ldots,2n\}\} $. We claim that these two cycles are disjoint. Suppose otherwise, there is an $ e \in E(C_{2n+1}) \cap E(C'_{2n+1}) $. Without loss of generality, assume one vertex on $ e $ is $ i $, then $ i~\blueedge~(i+1) \pmod{2n+1} $ and $ i~\rededge~(i+n) \pmod{2n+1} $, which implies that $ n-1 \equiv 0 \pmod{2n+1} $, a contradiction.

    \noindent Now we show that the diamond condition in \cref{lem 3} holds. Without loss of generality assume the edges of $ C $ are coloured by blue and the edges of $ C' $ are coloured by red in $ C\oplus C' $. Let $ i \in V(C\oplus C') $, then we have the following two diamonds in $ C\oplus C' $:
    \begin{align*}
        i~\blueedge~(i+1)~\rededge~(i+1+n)~\blueedge~(i+n)~\rededge~i,\\
        i~\blueedge~(i-1)~\rededge~(i-1-n)~\blueedge~(i-n)~\rededge~i.
    \end{align*}
This implies that the graph $ C_g\ast_hC'  $ contains the following two set of edges $ E_1 = \{i~\blueedge~(i+n-1) \pmod{2n+1}\mid i\in\{0,1,\ldots,2n\}\} $ and $ E_2 = \{i~\blueedge~(i+n+1) \pmod{2n+1}\mid i\in\{0,1,\ldots,2n\}\} $.
Now suppose that there are two vertices $ i,j \in V(C\oplus C')  $ such that they are on two different diamonds as follows:
\begin{align*}
    i~\blueedge~\ell_1~\rededge~j~\blueedge~\ell_1'~\rededge~i,\\
    i~\blueedge~\ell_2~\rededge~j~\blueedge~\ell_2'~\rededge~i\\
\end{align*}
for some $ \ell_1,\ell_1',\ell_2,\ell_2' \in V(C\oplus C') $. Without loss of generality assume $ \ell_1 = (i+1) $, then either $ j = (i+1+n) $ or $ j = (i+1-n) $. If $ 
j = (i+1+n) $, then $ \ell_1' = (i+n) $ and $ \ell_2' = (i+2+n) $. This implies that either $ (i+2+n)-n \equiv i \pmod{2n+1} $ or $ (i+2+n)+n \equiv i \pmod{2n+1} $, a contradiction. 
A similar argument works for $ j = (i+1-n) $. Hence the diamond condition in \cref{lem 3} holds and
\begin{align*}
    E(C_g\ast_hC') = E_1\sqcup E_2,
\end{align*}   
which implies that $ C_g\ast_hC' = Ci_{2n+1}(n-1,n) $, as desired. 
\end{proof}

\begin{exa}
Let $ C_{7} $ and $ C'_{7} $ be cycles with $ g\colon V(C_{7}) \to \{0,1,\ldots,6\} $ and $ h\colon V(C_{7}) \to \{0,1,\ldots,6\}  $, respectively, such that $ E(C_{7}) = \{i~\blueedge~(i+1) \pmod{7}\mid i\in\{0,1,\ldots,6\}\} $ and $ E(C'_{7}) = \{i~\rededge~(i+3) \pmod{7}\mid i\in\{0,1,\ldots,2n\}\} $. Then we can see that $ A(C_7)A(C_7') = A(Ci(2,3)) $:

\begin{align*}
    A(C_7)A(C_7') = \begin{bmatrix} 
	0 & 1 & 0 & 0 & 0 & 0 & 1 \\
	1 & 0 & 1 & 0 & 0 & 0 & 0 \\
    0 & 1 & 0 & 1 & 0 & 0 & 0 \\
	0 & 0 & 1 & 0 & 1 & 0 & 0 \\
 0 & 0 & 0 & 1 & 0 & 1 & 0  \\
 0 & 0 & 0 & 0 & 1 & 0 & 1 \\
  1 & 0 & 0 & 0 & 0 & 1 & 0 \\
	\end{bmatrix}
 \begin{bmatrix} 
	0 & 0 & 0 & 1 & 1 & 0 & 0 \\
	0 & 0 & 0 & 0 & 1 & 1 & 0 \\
    0 & 0 & 0 & 0 & 0 & 1 & 1 \\
	1 & 0 & 0 & 0 & 0 & 0 & 1 \\
 1 & 1 & 0 & 0 & 0 & 0 & 0  \\
 0 & 1 & 1 & 0 & 0 & 0 & 0 \\
  0 & 0 & 1 & 1 & 0 & 0 & 0 \\
	\end{bmatrix} = 
 \begin{bmatrix} 
	0 & 0 & 1 & 1 & 1 & 1 & 0 \\
	0 & 0 & 0 & 1 & 1 & 1 & 1 \\
    1 & 0 & 0 & 0 & 1 & 1 & 1 \\
	1 & 1 & 0 & 0 & 0 & 1 & 1 \\
 1 & 1 & 1 & 0 & 0 & 0 & 1  \\
 1 & 1 & 1 & 1 & 0 & 0 & 0 \\
  0 & 1 & 1 & 1 & 1 & 0 & 0 \\
	\end{bmatrix}.
\end{align*}
\end{exa}

\section{Conclusion}
\begin{table}[H]
  \centering{
  \resizebox{\textwidth}{!}{
    \begin{tabular}{p{12cm}p{1.45cm}ccl} \hline
    Graph class &  MPF &  Ref. \\ \hline
    $(4n+1)$-vertex complete graph, $K_{4n+1}$ &  \,\,\,\,\cmark & \Cref{thm:completegraphs} \\
    $K_{m}$, where $m\neq 4n+1$ & \,\,\,\,\xmark & \Cref{thm:completegraphs} \\
    $d$-dimensional hypercube, $Q_d$ &  \,\,\,\,\cmark & \Cref{thm:cube}\\
    $n$-vertex path, $P_{n}$ &  \,\,\,\,\xmark & \Cref{prop:path}  \\
    $2n$-vertex cycle, $C_{2n}$ &  \,\,\,\,\cmark & \Cref{cycle} \\
    $C_{2n+1}$ &  \,\,\,\,\xmark & \Cref{cycle} \\
    Complete bipartite graph, $K_{2m,2n}$ &  \,\,\,\,\cmark & \Cref{thm:completebipartitegraphs} \\
     $K_{2m+1,n}$ &  \,\,\,\,\xmark & \Cref{thm:completebipartitegraphs} \\
    $m \times n$- grid graph, $G_{m,n}$ $(m, n \ge 2)$ &  \,\,\,\,{\Large\textbf{?}} & {\Large\textbf{?}}  \\
    $m \times n$- torus graph, $C_n\Box C_m$ $(m, n \ge 2)$ &  \,\,\,\,{\Large\textbf{?}} & {\Large\textbf{?}}  \\
    Erd\H{o}s–R\'{e}nyi random graph $G(n, p)$  &  \,\,\,\,{\Large\textbf{?}} & {\Large\textbf{?}}  \\
    $(2n+1)$-vertex tree  &  \,\,\,\,{\Large\textbf{?}} & {\Large\textbf{?}}  \\
    Web graph $W_{2n,2}$ & \,\,\,\,\cmark & 
    \Cref{webgraph}  \\
    Web graph $W_{2n+1,2}$ & \,\,\,\,\xmark & 
    \Cref{C_odd}  \\
    Wheel graph $W_{1,n}$ & \,\,\,\,\xmark & 
    \Cref{wheel-graph} \\
     Generalized wheel graph $W_{2n,2m}$ & \,\,\,\,\cmark & 
    \cite[Remark 2.2]{Bhat} \\
    Friendship graph $F_{n}$ & \,\,\,\,\xmark & 
    \Cref{frined} \\
    Circulant graph $ Ci_{2n+1}(n-1,n)  $ &  \,\,\,\,\cmark & \Cref{circ}  \\
    \hline
    \end{tabular} 
    
  } 
  }
  \caption{Summary of our results.}
  \label{summary-table}
\end{table}

We close the paper with the following questions:
\begin{enumerate}
   
   \item Let $G$ and $H$ be two MPF graphs. Under what conditions is the Cartesian product $G \Box H$ also an MPF graph?
   \item What can we say about strongly regular graphs? Are they MPF?
\end{enumerate}

\bibliographystyle{plainurlnat}
\bibliography{MPF.bib}
\nocite{*}
\newpage
    
\end{document}